\documentclass{article}
\usepackage{amsthm}
\newtheorem{theorem}{Theorem}[section]

\newtheorem{lemma}[theorem]{Lemma}

\newtheorem{cor}[theorem]{Corollary}
{\theoremstyle{definition}
\newtheorem{defin}[theorem]{Definition}

\newtheorem{rem}[theorem]{Remark}

{\theoremstyle{definition}
\newtheorem{exam}[theorem]{Example}
\newcommand{\mycomment}[1]{}

\usepackage{amsmath,amssymb, amsfonts}
\usepackage[margin=1.5in]{geometry}
\usepackage{mathrsfs}
\usepackage{enumerate}
\usepackage{graphicx}
\graphicspath{ {./images/} }
\usepackage{nicefrac}
\usepackage{setspace}
\usepackage{relsize}
\usepackage{hyperref}
\usepackage[numbers,sort&compress]{natbib}

\DeclareMathOperator{\interior}{int}

\title{Convergence of pushforward measures for certain countably piecewise linear Markov maps}
\author{Zoltán Kalocsai \thanks{Supported by the Eköp-24 University Excellence Scholarship Program of the Ministry for Culture and Innovation from the source of the National Research, Development and Innovation Fund. Supported by ELTE Eötvös Loránd University, Budapest, Hungary Faculty of Science}}
\date{ }

\begin{document}

\maketitle

\begin{abstract}
We study piecewise linear Markov maps, with countable Markov partitions, inspired by a problem of the Miklós Schweitzer competition in 2022. We introduce $\ell$-Markov partitions and apply ideas of symbolic dynamics to our systems, relating them to Markov shifts. We survey how the Frobenius--Perron operators of these systems can be represented by matrices, and adapt results to countable alphabets. We apply these statements to prove a convergence theorem on the pushforwards of absolutely continuous measures. This enables us to prove a variety of useful ergodic properties of our maps and study even non-$\sigma$-finite absolutely continuous invariant measures. We explain how our results are not implied by previous ones and apply the convergence theorem to solve the original problem in the competition. 
\end{abstract}

\renewcommand{\thefootnote}{}

\tableofcontents

\section{Introduction}

Piecewise linear interval maps are one of the simplest discrete time dynamical systems. Thanks to their simple definition they are widely used for modelling phenomena, and by their research we can learn techniques widely applicable for more complex systems. Despite their apparent simplicity these systems can have extremely complex and unpredictable orbits. The study of the long-term stochastic behaviour of such constructions is thus relevant.\medskip

The absolutely continuous invariant measures, the convergence of pushforward measures and transfer operators for both linear and non-linear Markov interval maps are widely studied concepts. They are discussed in both textbooks and research papers such as \cite{Boyarsky, Lasota} and \cite{Bowen1979, sarigPF, Bugiel1985, Bugiel1991, Bugiel1998, Bugiel2021}. We introduce the necessary ideas for the discussion of these problems and expand on some of the known results in multiple ways.\medskip

In this work we consider dynamical properties of a class of real-to-real maps which have a countable Markov partition and are linear on the elements of the partition. These we call systems with an $\ell$-Markov partition. In Section \ref{sec1} we review the relevant theory of symbolic dynamics in our special case based on the notations and results of \cite{mineOfficial2}. We introduce the notion of expansivity, which is a weaker form of the 'expanding condition' appearing widely in literature, for example in \cite{Bugiel1998}. Statements on the topological conjugacy of the studied systems with shifts if expansivity is satisfied are listed.\medskip

In Section \ref{sec2} we describe how systems with $\ell$-Markov partitions relate to Markov chains with the help of the Frobenius--Perron operator. The core ideas in this section are generally known, the purpose of the section is to collect the relevant properties and show that they apply in our exact case as well. At the end of the section some relevant theorems on Markov chains are stated. With Corollary \ref{eloszlas} we lay the groundwork for our main results in the following section.\medskip

Most of our main, new results can be found in Section \ref{sec3}. Here we describe the absolutely continuous invariant measures of the studied systems. In Subsection \ref{convergence_theo} we state the main theorem, Theorem \ref{nagyhatar}. This result is unique in that it describes the convergence of the pushforward measures of all finite absolutely continuous measures for our systems. In Subsection \ref{prev_results} we describe the connection of this theorem to other known results with particular attention to Theorem 3.1 of \cite{Bugiel1998}. We show that Example \ref{savior} is a system for which our results apply, but previous known results do not.\medskip

In Subsection \ref{pos_recurrent} we prove some ergodic properties for systems associated with positively recurrent Markov chains. We describe all absolutely continuous invariant measures for these systems, considering infinite and non $\sigma$-finite measures as well, which are generally not considered in the literature.\medskip

The main inspiration behind this paper is the dynamical system presented in the third problem of the Miklós Schweitzer Competition 2022 organized by the János Bolyai Mathematical Society, proposed by my advisor Zoltán Buczolich. The website of the competition can be found at \cite{schweitz}. The problem was the following, with a bit modified notation:\medskip

\textit{
The transformation $T:\: [ 0,\infty ) \rightarrow [0, \infty)$ is linear on every positive integer interval, and its integer values are defined as follows:
\begin{equation}
T(n)=
	\begin{cases}
		0 & \text{if } 2|n\\
		4^{\ell}+1 & \text{if } 2 \nmid n,\, 4^{\ell - 1}\leq n < 4^{\ell} \: (\ell \in \mathbb{Z}^+). 
	\end{cases}
\end{equation}
Let $T^0(x)=x$ and $T^{n}(x) = T(T^{n-1}(x))$ for all $n>0$. Find the $\liminf_{k\rightarrow \infty}(T^k(x))$ and the $\limsup_{k\rightarrow \infty}(T^k(x))$ for Lebesgue almost every $x \in [0, \infty)$!\medskip
}

The proven Theorem \ref{nagyhatar} is a much stronger statement about wider class of piecewise linear transformations. This will imply that for almost every $x \in [0,\infty)$ the orbit is dense in $\mathbb{R}_{\geq 0}$, thus $\liminf_{k\rightarrow \infty}(T^k(x))=0$ and $\limsup_{k\rightarrow \infty}(T^k(x))=\infty$. We give our proof of these statements in Corollary \ref{solution}. In the last section the invariant measure of this specific system is also calculated. We list some preliminary definitions.\medskip

\begin{figure}[h]
\includegraphics[width= 10cm]{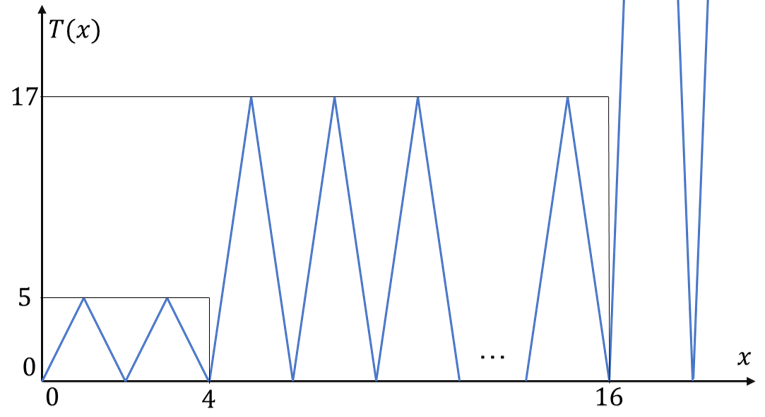}
\centering
\caption{The transformation $T$.}
\label{Fuggveny}
\end{figure}

\section{Preliminaries}
\begin{defin}
The function $\ell:\: \mathbb{R}_{\geq 0} \rightarrow \mathbb{N}$ is defined as follows:\[
\ell(x) = \inf\{n\in \mathbb{Z}^+:\: x \leq 4^n \}.
\]
\end{defin}

For a finite or infinite sequence $\underline{i} = (i_0, i_1, ...)$ we denote by $\underline{i}|n$ the finite sequence $(i_0, ..., i_n)$. For a Lebesgue-measurable subset $X$ of $\mathbb{R}^n$ we will denote the Lebesgue $\sigma$-algebra on $X$ with $\mathscr{L}(X)$. If a function $f$ is linear on the interior of an interval $H \subset \mathbb{R}$, then $f'(H)$ is $f'(x)$ for any $x\in \interior (H)$. For some sequence of reals $x_n$ and for some integers $a>b$ we define $\sum_{n=a}^b x_n = 0$ and $\prod_{n=a}^b x_n = 1$ as it is standard for the empty sum and product. Throughout this paper by almost every we mean Lebesgue almost every, and by measurable we mean Lebesgue-measurable. We denote the Lebesgue measure by $\lambda$. We say that a measure $\mu$ is absolutely continuous if $\lambda \gg \mu$. If $\mu$ is an absolutely continuous measure then by the density $\gamma$ of it we mean the function for which $\mu(H) = \int_H \gamma\, d\lambda$ for all $H \in \mathscr{L}(X)$, so $\gamma$ is its Radon-Nikodym derivative. The domain of the function $\gamma$ consists of the points of the measure space and its range is $\mathbb{R}_{\geq 0} \cup \{\infty \} = \overline{\mathbb{R}}_{\geq 0}$. By null set we mean sets of Lebesgue measure zero. We define $0\cdot \infty$ to be 0, so the measure of a null set is zero, even if the density is infinite everywhere.

\section{Preliminaries on symbolic dynamics} \label{sec1}
In this section we list some preliminary definitions and propositions explained in \cite{mineOfficial2}. Where appropriate we also cite relevant literature.
\subsection{$\ell$ - Markov partitions}
We are going to consider maps which possess $\ell$-Markov partitions defined in \cite{mineOfficial2}. Similar definitions also are present in \cite{Bowen1979, Bugiel1985, Bugiel1998}, and class $\mathcal{L}_M$ in section 9.1 of \cite{Boyarsky}. We use the notations of \cite{mineOfficial2}.\medskip

\begin{defin}[$\ell$-Markov partition, Definition 3.1. of \cite{mineOfficial2}] \label{Markovpart}
Given a measurable subset $X$ of the reals, which is a countable union of closed subintervals with non-empty interiors, and a measurable $f:\: X\rightarrow X$ the dynamical system $(X, f)$ has an $\ell$-Markov partition if there exists a countable set $\mathscr{H}$ of closed subintervals of $\mathbb{R}$ with non-empty interior such that the following are true:
\begin{enumerate}
\item For all $A, B \in \mathscr{H}$ if $A \neq B$ then $\interior (A) \cap \interior (B) = \emptyset$.
\item $X = \cup \mathscr{H}$.
\item For all $A \in \mathscr{H}$ there exists $\emptyset \neq \mathscr{A} \subseteq \mathscr{H}$ such that $\interior(f(A)) = \interior (\cup \mathscr{A})$.
\item The function $f$ is linear on the interior of all $H \in \mathscr{H}$.
\end{enumerate}
\end{defin}

\begin{defin}[Definition 3.2. of \cite{mineOfficial2}]
Let $(X,f)$ be a dynamical system with an $\ell$-Markov partition $\mathscr{H}$ as in Definition \ref{Markovpart}. Let the elements of $\mathscr{H}$ be $H_i$ where $i$ is a positive integer. Then for a sequence of positive integers of not neccessarily finite length $(i_0, i_1, ...)$ define:\[
H_{i_0i_1...} = \bigcap_{n=0}^{\infty} f^{-n}(H_{i_n}).
\]
If the sequence is finite then we only take finite intersections.
\end{defin}

\begin{defin}[Definition 3.3. of \cite{mineOfficial2}]
Let us define the set $D_0 = \{x \in X: \exists i \in \mathbb{Z}^+: x \in \partial H_i \}$. Let $D$ be the set $\bigcup_{k=0}^\infty f^{-k}(D_0)$. Let for all $A \subseteq X$ the set $\widetilde{A}$ be $A\setminus D$.
\end{defin}

\begin{cor}[Corollary 3.4. of \cite{mineOfficial2}]
The set $D$ is countable, thus $\lambda(A \Delta \widetilde{A})=0$ for all $A \in \mathscr{L}(X)$. On all $\widetilde{H_i}$ the function $f$ is a homeomorphism between $\widetilde{H_i}$ and $f(\widetilde{H_i})$.
\end{cor}

\begin{lemma}[Lemma 3.5. of \cite{mineOfficial2}] \label{Rendszer2}
Given a finite sequence of positive integers $(i_0,...,i_n)$ the following are true for $H_{i_0...i_n}$:
\begin{enumerate}
\item For all $H_i \in \mathscr{H}$ we have that $f(\interior (H_i)) = \interior (f(H_i))$.

\item $\interior (H_{i_0...i_n})$ is an interval.

\item If $\interior (H_i) \cap f^{-1}(H_j) \neq \emptyset$ then $\interior(H_j) \subseteq \interior (f(H_i))$.

\item If $H_{i_0...i_n}$ has a non-empty interior then $f^k(\interior (H_{i_0...i_n})) = \interior (H_{i_k...i_n})$ for all integers $1 \leq k\leq n$.

\item For all $(j_0...j_n) \neq (i_0...i_n)$ sequences of positive integers $\interior (H_{i_0...i_n}) \cap \interior (H_{j_0...j_n}) = \emptyset$.

\item $\interior (H_{i}) = \interior ( \bigcup \{ H_{ik}: \interior (H_k) \subseteq f(H_{i}) \})$.

\item $\interior (H_{i_0...i_n}) = \interior ( \bigcup \{ H_{i_0...i_nk}: \interior (H_k) \subseteq f(H_{i_n}) \})$.

\item For all $k\in \mathbb{Z}^+$ the intersection $\interior (H_{i_0...i_nk}) \cap  \interior (H_{i_0...i_n})$ is non-empty if and only if 
$k \in \{m:\interior(H_m) \subseteq f(H_{i_n}) \}$.

\item If $( \interior (H_{i_0...i_nk}) \cap \interior (H_{i_0...i_n}) ) \neq \emptyset$ then $\lambda(H_{i_0...i_nk}) = \frac{\lambda(H_{k})}{|f'(H_{i_n})|} \lambda(H_{i_0...i_n})$.
\end{enumerate}
\end{lemma}

\subsection{Expansivity}
For the most important properties, we need some stronger assumptions. We refer to \cite{mineOfficial2}.\medskip

For $A,B \subseteq X$ we use the notation $A \Delta B = (A \setminus B) \cup (B \setminus A)$ for the symmetric difference.

\begin{defin}[Expansivity according to Definition 3.6. in \cite{mineOfficial2}] \label{expansdef}
The dynamical system $(X,f)$ is said to have an expansive $\ell$-Markov partition, if for all $\delta > 0$ and $I \subseteq X$ bounded interval there is a finite set of positive integer sequences of finite length $(i^{(1)}_0,..., i^{(1)}_{n_1}), ... , (i^{(N)}_0,..., i^{(N)}_{n_N})$ for which: \[
\lambda \Big( I \Delta \bigcup_{k=1}^N H_{i^{(k)}_0...i^{(k)}_{n_k}} \Big) < \delta.
\]
\end{defin}

\begin{theorem}[Equivalent definition of expansive $\ell$-Markov partition according to Theorem 3.10. in \cite{mineOfficial2}] \label{expans2}
The $\ell$-Markov partition $\mathscr{H}$ of the dynamical system $(X,f)$ is expansive if and only if the following criterion holds. For all $\{ i_n \}_{n=0}^\infty \in (\mathbb{Z}_{> 0} )^\mathbb{N}$ such that $\tilde{H}_{i_0 i_1 ...} \neq \emptyset$ we have the following:
\begin{equation} \label{shrinking}
0 = \lim_{n\to \infty} \lambda (H_{i_0 i_1 ... i_n})
\end{equation}
\end{theorem}

We remark that if the measure of the sets of the partition is bounded and there is some $\varepsilon > 0$ that $|f'(H_i)|>1+\varepsilon$ for all $H_i \in \mathscr{H}$, then $\mathscr{H}$ is expansive.
 
\begin{defin}[Definition 3.7. of \cite{mineOfficial2}]
Let the symbolic topology $\tau_\sigma$ on $\widetilde{X}$ be the one generated by the empty set, and the elements of $\{\widetilde{H}_{i_0...i_n}: n \in \mathbb{Z}^+, (i_0,...,i_n)\in (\mathbb{Z}^+)^n\}$.
\end{defin}

\begin{theorem}[Theorem 3.8. of \cite{mineOfficial2}] \label{topologia}
If the $\ell$-Markov partition is expansive, then the symbolic topology is the standard subspace topology $\tau$ on $\widetilde{X}$ inherited from $\mathbb{R}$.
\end{theorem}

\begin{defin}[Definition 3.12. of \cite{mineOfficial2}]
Let the function $\underline{i}: \widetilde{X} \rightarrow \mathbb{Z}^\mathbb{N}$ be the itinerary of a point. For a point $x\in \widetilde{X}$ it gives the integer sequence $(i_0, ...)$ for which $x \in \widetilde{H}_{i_0...}$. 
\end{defin}

\begin{defin}[Definition 3.13. of \cite{mineOfficial2}] \label{trajspace}
We denote by $\Omega_f = \{\underline{i}(x):\: x \in \widetilde{X}\}$ the space of allowed symbolic trajectories. We define the metric $d: \Omega_f \times \Omega_f \rightarrow \widetilde{X}$ such, that $d((i_0,...), (j_0,...)) = 2^{-n}$ where $n$ is the smallest non-negative integer for which $i_n \neq j_n$.
\end{defin}

For all the remaining statements in this subsection it is assumed that we have an expansive $\ell$-Markov partition.

\begin{theorem}[Theorem 3.14. of \cite{mineOfficial2}] \label{Homeo}
The function $\underline{i}$ is a homeomorphism between $\widetilde{X}$ and $\underline{i}(\widetilde{X}) = \Omega_f$.
\end{theorem}

\begin{cor}[Corollary 3.15. of \cite{mineOfficial2}] \label{conju}
The dynamical systems $(\widetilde{X}, f)$ and $(\Omega_f, \sigma)$ are topologically conjugate.
\end{cor}

\begin{theorem}[Theorem 3.16. of \cite{mineOfficial2}]
If for some $x \in \widetilde{X}$ the orbit of $\underline{i}(x)$ is periodic for the shift operator then the orbit of $x$ is periodic as well.
\end{theorem}

\subsection{Application to $(\mathbb{R}_{\geq 0}, T)$}

We list the statements in \cite{mineOfficial2} on $(\mathbb{R}_{\geq 0}, T)$. We call closed integer intervals the sets $I_n = [n-1, n]$ for any $n\in \mathbb{Z}^+$.

\begin{defin}[Definition 3.18. of \cite{mineOfficial2}] \label{symbolinterval}
Given a finite sequence of positive integers $(i_0,...,i_N)$ we denote $I_{i_0...i_N} = \bigcap_{n=0}^{N} T^{-n}(I_{i_n})$. Given an infinite sequence of positive integers $(i_0,...)$ we denote 
$I_{i_0...} = \bigcap_{n=0}^{\infty} T^{-n}(I_{i_n})$.
\end{defin}

\begin{lemma}[Lemma 3.19. of \cite{mineOfficial2}]\label{Rendszer}
Given a finite sequence of positive integers $(i_0,...i_n)$ the following are true for $I_{i_0...i_n}$:
\begin{enumerate}
\item The sets $I_{i_0...i_nk}$ where $k \in \mathbb{Z} \cap [1,4^{\ell(i_n)}+1]$ are placed in a row in increasing or decreasing order. (See Figure \ref{rendabra}.)

\item The above sets are in increasing order if and only if there is an even number of even numbers among $i_0 ... i_n$. (See Figure \ref{rendabra}.)
\end{enumerate}
\end{lemma}

\begin{figure}[h]
\includegraphics[width= 10cm]{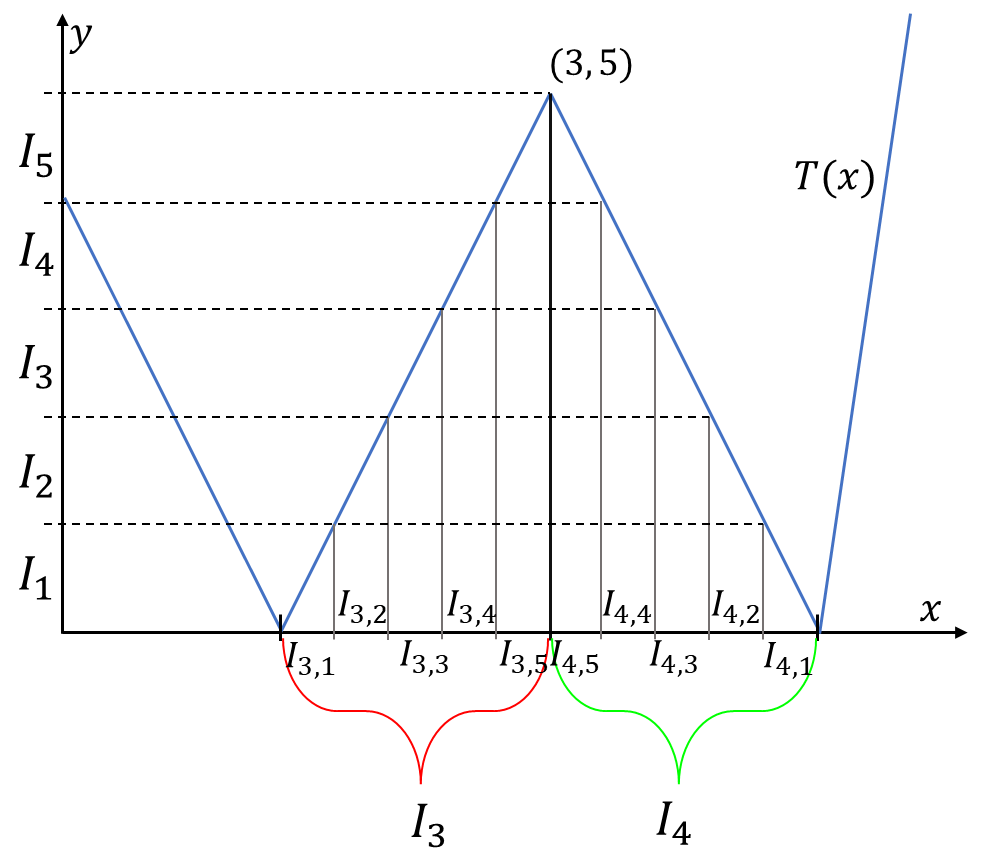}
\centering
\caption{Some labelled intervals of $\mathbb{R}_{\geq 0}$.}
\label{rendabra}
\end{figure}

\begin{theorem} [Theorem 3.20. of \cite{mineOfficial2}]
The set of integer intervals forms an expansive $\ell$-Markov partition of $(\mathbb{R}_{\geq 0}, T)$.
\end{theorem}

\begin{theorem} [Theorem 3.21. of \cite{mineOfficial2}]
The space $\Omega_T$ is composed of exactly those sequences $(i_0, i_1, ...) \in (\mathbb{Z}_{>0})^\mathbb{N}$, which satisfy the following conditions:
\begin{enumerate}
\item For all $n > 0$ we have that $i_{n+1} \in \{1, 4^{\ell(i_n)}+1\}$
\item There exists no $n \in \mathbb{N}$ such that $\sigma^n((i_0, i_1, ...)) = (1,1,1,...)$ or $\sigma^n((i_0, i_1, ...)) = (4^{\ell(i_{n-1})}+1, 4^{\ell(i_{n-1})+1}+1, 4^{\ell(i_{n-1})+2}+1, ...)$ 
\end{enumerate}
\end{theorem}

\section{Frobenius--Perron operator, $\ell$-Markov partitions and Markov chains} \label{sec2}

In this section we discuss the matrix representation of the Frobenious-Perron operator for maps with an $\ell$-Markov partition. We also discuss how $\ell$-Markov partitions relate to Markov-chains through this representation. At the end of the chapter we cite relevant results on Markov chains. Most ideas collected in this section are present in literature in similar forms, as they will be cited accordingly. The aim of this chapter is to collect all the relevant information in the form we want to apply it.

\begin{defin}
Given a measure space $(X, \Sigma, \mu)$ and a map $f:\: X \rightarrow X$, the pushforward of $\mu$ is $\mu \circ f^{-1}$.
\end{defin}

\begin{defin}[Assumptions] \label{majdnemdif}
Let our measure space $X$ be some countable union of closed subintervals of the reals and the $\sigma$-algebra be the corresponding Lebesgue $\sigma$-algebra. We assume that $f$ is differentiable everywhere but at countably many points. Given a density $\rho :\: X \rightarrow \overline{\mathbb{R}}_{\geq 0}$ let $\mu(H) = \int_H \rho\, d\lambda$ for all $H \in \mathscr{L}( X )$ meaning that $\mu \ll \lambda$. In that case the $(\mu \circ f^{-1})(H) = \int_{f^{-1}(H)} \rho\, d\lambda$.
\end{defin}

\begin{defin}
Let $D(f)$ be the domain of the function $f$ and $R(f)$ be its range. We suppose that $D(f), R(f) \subseteq \mathbb{R}$. The Frobenius--Perron operator for transformation $f:\: D(f) \rightarrow R(f)$ maps the space of measurable functions $\mathbb{R} \rightarrow \overline{\mathbb{R}}_{\geq 0}$ which are supported on $D(f)$ into the space of measurable functions $\mathbb{R} \rightarrow \overline{\mathbb{R}}_{\geq 0}$, which are supported on $R(f)$. The operator is denoted by $\mathscr{P}_f$. If $\mu(H) = \int_H \rho d\lambda$, then the function $\mathscr{P}_f(\rho)$ is the density of the pushforward measure: \[
\mu(f^{-1}(H)) = \int_{f^{-1}(H)} \rho\, d\lambda = \int_{H} \mathscr{P}_f(\rho)\, d\lambda .
\]
\end{defin}

Let $D(f)$ be an interval. If $f$ is a diffeomorphism from its domain to its range, then the operator on the range can be represented in a more compact way:
\begin{equation} \label{egyszepPerron}
\mathscr{P}_f(\rho)(x) = \frac{\rho(f^{-1}(x))}{|f'(f^{-1}(x))|}.
\end{equation} 
This representation comes from the formula for change of variables. It is widely used in literature. For example it is used on page 86. of \cite{Boyarsky} and it is stated 
by equation (3.2.7) of \cite{Lasota}.\medskip 

We consider functions, that have an $\ell$-Markov partition, for which $\lambda(H_i) < \infty$ for all $H_i \in \mathscr{H}$. If we restrict such a function to $\interior (H)$ for an $H \in \mathscr{H}$ we get that $f|_{\interior (H)}$ is a diffeomorphism between $\interior (H)$ and $f(\interior (H))$. Let $\mathscr{H} = \{H_i\}_{i\in \mathbb{Z}^+}$. Suppose that $\mu(A) = \int_A \rho d\lambda$ for all $A\in \mathscr{L}(X)$. Then from the $\sigma$-additivity of measures we have, that $\mu(A) = \sum_{i=1}^\infty \mu(A \cap H_i)$, and since $\mu$ is absolutely continuous and $\mu(H_i) = \mu(\interior (H_i))$, thus $\mu(A) = \sum_{i=1}^\infty \mu(A \cap \interior (H_i))$. Thus we have the following:
\begin{equation}
\mu(f^{-1}(A)) = \sum_{i=1}^\infty \int_{f^{-1}(A)\cap \interior (H_i)} \rho\, d\lambda = \sum_{i=1}^\infty \int_{(f|_{\interior (H_i)})^{-1}(A)} \rho \cdot \chi_{\interior (H_i)}\, d\lambda 
\end{equation}
\begin{equation}
\mu(f^{-1}(A)) = \sum_{i=1}^\infty \int_A \mathscr{P}_{f|_{\interior (H_i)}}(\rho \cdot \chi_{\interior (H_i)})\, d\lambda.
\end{equation}

According to Beppo Levi's theorem as stated in Corollary 21.39.b. of \cite{Eric}, we have that:
\begin{equation} \label{aaaqa}
\mu(f^{-1}(A)) = \sum_{i=1}^\infty \int_A \mathscr{P}_{f|_{\interior (H_i)}}(\rho \cdot \chi_{\interior (H_i)}))\, d\lambda = \int_A \sum_{i=1}^\infty \mathscr{P}_{f|_{\interior (H_i)}}(\rho \cdot \chi_{\interior (H_i)}))\, d\lambda.
\end{equation}
We have that $f|_{\interior (H_i)}$ is a diffeomorphism between its domain and its range. This along with \eqref{egyszepPerron} and  \eqref{aaaqa} implies the following:
\begin{equation} \label{szepPerron}
\mathscr{P}_f(\rho)(x) = \sum_{i=1}^\infty \mathscr{P}_{f|_{\interior (H_i)}}(\rho \cdot \chi_{\interior (H_i)})(x) = \sum_{z \in f^{-1}(x)} \frac{\rho(z)}{|f'(z)|}.
\end{equation}

This representation for finite partitions is present in (4.3.4) in \cite{Boyarsky}, however with the use of Beppo Levi's theorem it generalizes in a straightforward way to our countably infinite case, as demonstrated before.\medskip

This representation does not make sense if $f^{-1}(x)$ has points where $f$ is not differentiable. We assumed in Definition \ref{majdnemdif} that $f$ is differentiable at everywhere but countably many points. Since for all $z\in X$, if $z \in f^{-1}(x) \cap f^{-1}(y)$, where $x,y\in X$, then $x=y$. This means that the points, where $f$ is not differentiable can only be included in the preimages of countably many points, thus \eqref{szepPerron} defines the new density almost everywhere, thus its integral will be well-defined.\medskip

According to Proposition 4.2.1, 4.2.3 and 4.2.4 of \cite{Boyarsky} the Frobenious-Perron operator is linear, the image of the density of a probability measure is itself a density of a probability measure and the operator is continuous in the $L_1$ norm. The proofs of these statements work word for word in our case as well.

\begin{defin} \label{Transmatrix}
Let $(X, f)$ be a dynamical system with an $\ell$-Markov partition $\{H_k\}_{k\in \mathbb{Z}^+}= \mathscr{H}$, such that for all $H_k \in \mathscr{H}$ we have that $\lambda(H_k) < \infty$.  We say that $P = (p_{ij})$ is its transition matrix if the following are satisfied:
\begin{enumerate}
\item If $n = |\mathscr{H}|$ then $P$ is an $n\times n$ matrix, where $n\in \mathbb{Z^+}$. If $\mathscr{H}$ is countably infinite, then $P=(p_{ij})_{i,j\in \mathbb{Z^+}}$.
\item For $i,j$ where $\interior (H_j) \subseteq f(H_i)$ it is true that 
\begin{equation} \label{transprob}
p_{ij} = \frac{\lambda(H_j)}{\lambda(H_i) |f'(H_i)|}
\end{equation}
where $f'(H_i)$ is the slope of $f$ on the interior of $H_i$. Otherwise $p_{ij}$ is zero. 
\end{enumerate}
\end{defin}

For most statements we assume that a partition is countably infinite. In all cases the statements and their proofs can trivially be modified for finite partitions.

\begin{defin}
An infinite, non-negative matrix is row stochastic if for any row the sum of the first $n$ elements of it tends to one as $n$ tends to infinity.
\end{defin}

\begin{cor}
The transition matrix is always row-stochastic.
\end{cor}

\begin{defin}
For an $n \in \mathbb{Z}^+ \cup \{\infty\}$ dimensional row vector $v$, let $D_v$ be the $n\times n$ dimensional diagonal matrix with $v$ in its diagonal. In other words, the $n\times n$ matrix for which $[D_v]_{ij} = 0$ if $i \neq j$ and $[D_v]_{ii}=v_i$.
\end{defin}

\begin{defin}
Suppose that we have an absolutely continuous measure $\mu$ with a possibly infinite density $\rho$ and it is constant $c_i \in \overline{\mathbb{R}}_{\geq 0}$ on the interiors of all $H_i \in \mathscr{H}$. Let us denote the row vector $[c_1,c_2,...]$ by $[\rho]$.
\end{defin}

\begin{theorem} \label{perron}
Suppose that we have an absolutely continuous measure $\mu$ with a possibly infinite density $\rho$ and it is constant $c_i$ on the interiors of all $H_i \in \mathscr{H}$. Furthermore suppose, that for all $H_i \in \mathscr{H}$ we have that $\lambda(H_i) < \infty$. Let us denote the row vector $[\lambda (H_1), \lambda (H_2),...]$ by $\Lambda$ and let us denote $[\rho]$ by $C$.
In this case for all $k\in \mathbb{N}$:
\begin{equation} \label{markovVektor0}
\mathscr{P}_f^k(\rho)(x) = (C D_{\Lambda} P^k D_{\Lambda}^{-1})_j \text{ for all } x \in \interior (H_j).
\end{equation}
\end{theorem}

\begin{rem}
If we assume that the partition is finite and that $\cup \mathscr{H}$ is an interval, then this statement is equivalent to Theorem 9.2.1 of \cite{Boyarsky}. According to that notation let $M_f := D_{\Lambda} P D_{\Lambda}^{-1}$, thus $M_f^k = D_{\Lambda} P^k D_{\Lambda}^{-1}$.
\end{rem}

\begin{proof}
First let us assume that $\rho = 1$ on $\interior (H_i)$ for some $i$ and zero otherwise. Let us start with $k=1$. Since we have linearity on all $H_n$, we have that $f^{-1}(x) \cap H_n$ has at most one element, for all $H_n \in \mathscr{H}$. As a consequence of this the left hand side of \eqref{markovVektor0} for all $x \in \interior (H_j)$ equals:\[
\mathscr{P}_f(\rho)(x) = \sum_{z \in f^{-1}(x)} \frac{c_i \chi_{\interior (H_i)}(z)}{|f'(H_i)|} =
\begin{cases}
		0 & \text{if } p_{ij}=0\\
		\frac{c_i}{|f'(H_i)|} & \text{otherwise.}
	\end{cases}
\]
The right hand side from \eqref{transprob} is as follows: \[
\frac{c_i \lambda(H_i) p_{ij}}{\lambda(H_j)} =  \begin{cases}
		0 & \text{if } p_{ij}=0\\
		\frac{c_i}{|f'(H_i)|} & \text{otherwise.}
	\end{cases}
\]

Now we can prove the statement for $k=1$ by noting that both sides of \eqref{markovVektor0} are linear in terms of $C$. Let us consider the densities $\{\rho_n \}_{n\in \mathbb{Z}^+}$. Their definition is as follows:
\[
\rho_n(x) = 
\begin{cases}
		1 & \text{if } x \in \interior (H_n)\\
		0 & \text{otherwise.}
	\end{cases}
\]

Let us denote $[\rho_n]$ by $C_n$, so all coordinates of $C_n$ are zero except for the $n$th which is one. Any density $\rho$ and vector $C$ we are considering can be written in the following way:\[
\rho = \sum_{n=1}^\infty \alpha_n \rho_n, \quad C = \sum_{n=1}^\infty \alpha_n C_n \text{ where } \alpha_n \in \overline{\mathbb{R}}_{\geq 0}.
\]

Note that since all summands are non-negative these sums always exist in $\overline{\mathbb{R}}$. From linearity of the formulae we have the following:
\begin{equation} \label{k1}
\Big[ \mathscr{P}_f(\rho) \Big] = \sum_{n=1}^\infty \alpha_n \Big[\mathscr{P}_f(\rho_n) \Big] = \sum_{n=1}^\infty \alpha_n C_n M_f = C M_f.
\end{equation}

This proves the statement for $k=1$. It is also clear that for a density $\rho$ satisfying the assumptions of the theorem the density $\mathscr{P}_f(\rho)$ also satisfies them. Suppose that for all $k < n$ the statement holds, then for $n$:
$$\Big[ \mathscr{P}^n_f(\rho) \Big] = \Big[ \mathscr{P}^{n-1}_f(\mathscr{P}_f(\rho)) \Big] = \Big[ \mathscr{P}_f(\rho) \Big]M_f^{n-1}= (C M_f)M_f^{n-1} = C M_f^n.$$
Thus the statement follows.
\end{proof}

This basically relates a dynamical system with an $\ell$-Markov partition to a Markov chain, hence the name of this property. This is extremely useful in finding absolutely continuous invariant measures of a system. The idea of using piecewise linear Markov maps to represent Markov-chains is present for example in \cite{Oohama}.

\begin{cor}
A row-eigenvector $v$ of $P$ associated with the eigenvalue 1 defines (almost everywhere) an invariant absolutely continuous measure $\mu$ of $f$ with density $\rho$ in the following way:\[
\rho(x) = \frac{v_i}{\lambda(H_i)} \text{ for all } x\in H_i.
\]
\end{cor}

The transition matrix defines a Markov chain with countable states. The elements of the matrix are called the transition probabilities. About the convergence of these values \cite{Chung} writes the following:

\begin{defin}
The following are parameters of the Markov chain defined by $P$:
\begin{enumerate}
\item $p_{ij}^{(k)}=(P^k)_{i,j}$ the probability of stepping on $j$ after $k$ steps if we are in $i$.
\item $f_{ij}^{(n)} = \sum_{s \in S} p_{is_1}p_{s_{n-1}j}\prod_{k=2}^{n-1}p_{s_{k-1}s_k}$ where $s = (s_1, ..., s_{n-1})$ and $S$ is the set of $n-1$ long positive integer sequences not containing $j$. This is the probability of starting from $i$ and taking a path arriving at $j$ for the first time after exactly $n$ steps.
\item $f_{ij}^* = \sum_{n=1}^{\infty} f_{ij}^{(n)}$ the probability of eventually getting to $j$ if we start from $i$.
\item $m_i = \sum_{n=1}^{\infty} nf_{ii}^{(n)}$ the average return time to $i$ if we start from it. For $i$ transient or null-recurrent state this is defined to be $\infty$.
\item $d_i = \gcd\{n: p_{ii}^{(n)}\neq 0\}$ the period of state $i$, where $\gcd$ stands for greatest common divisor.
\item $f_{ij}^*(r) = \sum_{n=1}^{\infty} f_{ij}^{(nd_i+r)}$. This is the probability of eventually getting to $j$ if we start from $i$ given that we only consider every $(nd_i+r)$th step.
\end{enumerate}
\end{defin}

\begin{theorem}[According to Theorem 4 on page 33 of \cite{Chung}] \label{Villo}
For all $r \in \{1,...d_j\}$ \[ \lim_{n\rightarrow \infty} p_{ij}^{(nd_j+r)}=f_{ij}^*(r) \frac{d_j}{m_j}.
\]
If $m_j = \infty$, then the limit is zero.
\end{theorem}

\begin{theorem}[According to Theorem 1.7.7 of \cite{Norris_1997}]\label{pozrec_cited_theo}
Let $P$ be the transition matrix of an irreducible Markov chain. The following are equivalent:
\begin{enumerate}
\item Every state is positive recurrent.
\item Some state $i$ is positive recurrent.
\item There exists a stochastic row vector $\Pi$ which is a stationary distribution. That is $\Pi = \Pi P$. We also have that $m_i = \frac{1}{\Pi_i}$
\end{enumerate}
\end{theorem}

\begin{cor} \label{kovetkeztetes}
If $\Pi$ is a finite or countably infinite dimensional row vector, such that $\Pi_i = \lim_{k \rightarrow \infty} p_{ii}^{(k)}$ then if the Markov chain is positive recurrent, then $\Pi$ is stochastic.
\end{cor}
\begin{proof}
This follows in a straightforward way from Theorem \ref{Villo} and Theorem \ref{pozrec_cited_theo}.
\end{proof}

\begin{cor} \label{Hatar}
If the matrix $P$ defines an aperiodic irreducible Markov chain, then the pointwise limit of $P^{k}$ exists as $k$ tends to infinity. If additionally $f^*_{ii} = 1$ for all $i$ then all the columns of $\lim_{k\rightarrow \infty} P^k$ are constant, so there is a row vector $w$ such that for all stochastic row vectors $v$ of sufficient dimension $\lim_{k\rightarrow \infty} vP^k = w$ coordinatewise. 
\end{cor}

\begin{proof}
If $P$ is aperiodic, then $d_j = 1$ for all $j$, hence by Theorem \ref{Villo} $\lim_{n\rightarrow \infty} p_{ij}^{(n)}$ exists for all $i,j$. The value of $d_j$ and $m_j$ only depend on the value of $j$ so they are constant throughout a column. With irreducibility we have that $p_{ij}^{(n)}\neq 0$ for all $i,j$ for some $n$. Since $f^*_{ii} = 1$ we have that we return infinitely many times from $i$ to $i$ with probability one. From this $f_{ij}^* =1$ for all $i, j$. This makes $\lim_{n\rightarrow \infty} p_{ij}^{(n)}$ independent of $i$. By this for all $v$ stochastic vectors $\lim_{k\rightarrow \infty} (vP^k)_j = \lim_{k\rightarrow \infty} p_{1j}^{(k)}$.
\end{proof}

\begin{cor} \label{eloszlas}
Suppose that $P$ satisfies all conditions in Corollary \ref{Hatar}. Let $\pi :\: \mathscr{L}(X) \rightarrow \overline{\mathbb{R}}_{\geq 0} $ be the absolutely continuous measure with density $\tau$ where:\[
\tau(x) = \frac{\lim_{k\rightarrow \infty} p_{ii}^{(k)}}{\lambda(H_i)} \text{ for all } x \in \interior (H_i).
\]
If $\mu:\: \mathscr{L}(X)\rightarrow \overline{\mathbb{R}}_{\geq 0} $ is an absolutely continuous measure with density $\rho$ where $\mu(X) < \infty$ and $\rho$ is constant on $\interior (H_i)$ for all $i$, then: 
\begin{enumerate}
\item If the associated Markov chain is not necessarily positive recurrent, then for all $A \in \mathscr{L}(X)$ such that there exists $N \in \mathbb{N}$, such that $\lambda (A \setminus \cup_{i=1}^N H_i ) = 0$ we have that $\lim_{k\rightarrow \infty}\mu(f^{-k}(A)) = \mu(X) \pi(A).$
\item If the associated Markov chain is positive recurrent, then $\mathscr{P}_f^k (\rho) \to \mu(X)\tau$ in $L^1$ as $k \to \infty$. Thus for all $A \in \mathscr{L}(X)$ we have that $\lim_{k\rightarrow \infty}\mu(f^{-k}(A)) = \mu(X) \pi(A).$
\end{enumerate}
\end{cor}

\begin{proof}
Let the vector of constants of $\rho$ be as defined in Theorem \ref{perron} be $C = [c_1, c_2...]$. By Theorem \ref{perron} we have that $$\mathscr{P}_f^k(\rho)(x) = (C M_f^k)_i \text{ for all } x \in \interior (H_i).$$
Let $r = \sum_{i=1}^\infty c_i \lambda(H_i) = \mu(X)$. Since we have assumed that $\mu(X)<\infty$ we have that $r<\infty$. Let $q = \frac{1}{r} (C D_\Lambda)$, thus $q$ is stochastic. By Corollary \ref{Hatar} we have the following:\[
\lim_{k \rightarrow \infty} ((C D_\Lambda)P^k)_i = r \cdot \lim_{k \rightarrow \infty} (qP^k)_i = r \cdot \lim_{k \rightarrow \infty} p_{ii}^{(k)}.
\]
This implies that on all $\interior(H_i)$ the sequence $\mathscr{P}_f^k(\rho)$ converges uniformly to $\mu(X)\tau$. From this point we treat the two cases separately.\medskip

For the first case we have assumed, that there exists $N$ such that $\lambda ( A \setminus \cup_{i=1}^N H_i) = 0$. Let us take $G = \cup_{i=1}^N H_i$. On $G$ the sequence $\mathscr{P}_f^k(\rho)$ converges uniformly to $\mu(X)\tau$, and $\lambda(G)<\infty$ thus $\mu(f^{-k}(A)) = \mu(f^{-k}(A \cap G)) \rightarrow \mu(X)\pi(A \cap G) = \mu(X)\pi(A)$. This proves the first statement.\medskip

In the positive recurrent case according to Corollary \ref{kovetkeztetes} we have the following. If $\Pi$ is a finite or countably infinite dimensional row vector, such that $\Pi_i = \lim_{k \rightarrow \infty} p_{ii}^{(k)}$, then $\Pi$ is stochastic, thus it is a stationary distribution of the Markov chain.\medskip

Given $\varepsilon > 0$ we give an $n_0$, such that for all $n > n_0$ we have that $\int_\mathbb{R} |\mathscr{P}_f^n(\rho) - \mu(X)\tau| d\lambda < \varepsilon$. Without loss of generality we may assume that $\mu(X) = 1$. Let us pick $N \in \mathbb{Z}^+$ such that $\sum_{i=1}^N \pi (H_i) = \sum_{i=1}^N \Pi_i > 1 - \frac{\varepsilon}{4}$. Let us pick $n_0$ such that for all $n>n_0$ we have $\sum_{i=1}^N \int_{H_i} |\mathscr{P}_f^n(\rho) - \tau| d\lambda < \frac{\varepsilon}{4}$. There exists such an $n_0$ as on any finite union of states $\mathscr{P}_f^n(\rho)$ converges uniformly to $\tau$ and any finite union of states has finite measure. From this for any $n>n_0$ we have that $\sum_{i=1}^N \int_{H_i} \mathscr{P}_f^n(\rho) d\lambda > 1- \frac{2\varepsilon}{4}$. Since $\int_{\mathbb{R}} \mathscr{P}_f^n(\rho) d\lambda= \int_{\mathbb{R}} \rho d\lambda$ we can conclude that $\sum_{i=N+1}^\infty \int_{H_i} | \mathscr{P}_f^n(\rho)| d\lambda < \frac{2\varepsilon}{4}$ for all $n > n_0$. Similarly we have that  $\sum_{i=N+1}^\infty \int_{H_i} | \tau | d\lambda < \frac{\varepsilon}{4}$. As a consequence of this $\sum_{i=N+1}^\infty \int_{H_i} | \mathscr{P}_f^n(\rho) - \tau | d\lambda < \frac{3\varepsilon}{4}$. As a consequence of this $\int_{\mathbb{R}} | \mathscr{P}_f^n(\rho) - \tau | d\lambda < \varepsilon$ for all $n > n_0$. This also implies that for all $A \in \mathscr{L}(X)$ we have that $\mu(f^{-k}(A)) \to \mu(X)\pi(A)$ as $k \to \infty$. This proves the second statement.

\end{proof}

\begin{cor} \label{stoch}
For the measure $\pi$ the measure of the whole space is $0 \leq \pi(X) \leq 1$.
\end{cor}

\section{Convergence and uniqueness of invariant measures for expansive $\ell$-Markov partitions} \label{sec3}

In this section we restrict the class of systems considered so that we can meaningfully discuss the absolutely continuous invariant measures. These will be systems with an expansive $\ell$-Markov partition. We prove that under certain conditions the pushforwards of all absolutely continuous probability measures converge to the same invariant measure. One of the corollaries of our results is that all of the finite absolutely continuous invariant measures of these systems are piecewise constant on the elements of the expansive $\ell$-Markov partition. For finite partitions this result is present in Theorem 9.4.2 of \cite{Boyarsky} with a different proof. Our results are unique in that we focus on the convergence of the pushforwards of all absolutely continuous measures, thus we prove stronger statements than just the characterization of the absolutely continuous invariant measures. Generally when discussing absolutely continuous invariant measures only $\sigma$-finite ones are considered however we consider any measure which has a $\rho: X \rightarrow \overline{\mathbb{R}}_{\geq 0}$ density. The convergence of the Frobenious-Perron operator for Markov maps is thoroughly studied in \cite{Bugiel1991, Bugiel1998, Bugiel2021} with assumptions differing from our investigation. At the end of the chapter we discuss how these theorems relate to our result.\medskip

We also refer to Theorem 4.9. of \cite{sarigPF}. In contrast to the cited theorem our results treat the convergence of pushforwards of any absolutely continuous invariant measures. In our case the cited theorem is only applicable to measures which have a density function $\chi_{H_{i_0...i_n}}$ for some positive integer sequence $i_0, ..., i_n$.
\subsection{Convergence theorem}\label{convergence_theo}

\begin{theorem} \label{nagyhatar}
Suppose that we have a dynamical system $(X, f)$ which has an expansive $\ell$-Markov partition $\mathscr{H}$, defining an aperiodic irreducible Markov chain which satisfies the conditions of Corollary \ref{Hatar}. Let $\pi :\: \mathscr{L}(X) \rightarrow \overline{\mathbb{R}}_{\geq 0} $ be the absolutely continuous stationary measure with density $\tau$ defined in Corollary \ref{eloszlas}. Let $\mu:\: \mathscr{L}(X) \rightarrow \overline{\mathbb{R}}_{\geq 0}$ be an absolutely continuous measure with $\mu(X) < \infty$. 
\begin{enumerate}
\item Let $A$ be a set in $\mathscr{L}(X)$ such that there exists $N \in \mathbb{N}$, such that $\lambda (A \setminus \cup_{i=1}^N H_i ) = 0$. In this case $\lim_{n\rightarrow \infty}\mu(f^{-n}(A)) =\mu(X)\pi(A)$.
\item If the Markov chain defined by the system is positive recurrent, then for all $A \in \mathscr{L}(X)$ we have that $\lim_{n\rightarrow \infty}\mu(f^{-n}(A)) =\mu(X)\pi(A)$. 
\end{enumerate}
\end{theorem} 

\begin{proof}
Let the density of $\mu$ be $\gamma$, such that $\mu(H)=\int_H \gamma d\lambda$. Let us approximate $\gamma$ by a simple function $\phi$ in the following way: \[
\phi = \sum_{k=1}^n c_k \chi_{E_k}, \quad \bigg| \int_H \phi \, d\lambda - \int_H \gamma \, d\lambda \bigg| < \varepsilon, \quad \forall H \in \mathscr{L}( X ).
\]

Here $E_k$ are measurable subsets of $X$ and $\chi_{E_k}$ is the characteristic function of $E_k$. Now we approximate the set $E_k$ by the union of finitely many closed intervals $R$ such that $\lambda(E_k \Delta R) < \delta$. 

\begin{lemma} \label{interaprox}
For all $\delta > 0$ and $E_k$ there is a union of finitely many closed intervals $R$ such that $\lambda(E_k \Delta R) < \delta$.
\end{lemma}
\begin{proof}
From the definition of the Lebesgue outer measure we can pick an $\{R_i\}_{i\in \mathbb{Z^+}}$ cover of $E_k$ with countably many closed intervals for which $E_k \subseteq \bigcup_{i\in \mathbb{Z^+}} R_i$ and $\sum_{i\in \mathbb{Z^+}} \lambda(R_i) < \lambda(E_k) + \frac{\delta}{2}$. In this case there is a finite union $R = \cup \{R_{i_k}\}_{k = 1}^N$ for which $R \Delta \bigcup_{i\in \mathbb{Z^+}} R_i < \frac{\delta}{2}$, which implies that $\lambda(E_k \Delta R) < \delta$.
\end{proof}

By Lemma \ref{interaprox} $\gamma$ can be approximated arbitrarily well by a function which is the linear combination of characteristic functions of finitely many closed intervals. Since $\mathscr{H}$ is expansive we can do it with intervals labelled by finite sequences as well.\medskip

Let us pick such an approximation $\omega$ of $\gamma$. Let $\Omega(H) = \int_H \omega\,d\lambda$. In this case the measure $\Omega$ is an approximation of $\mu$. For all $\varepsilon > 0$ we can pick $\omega$ such that  $| \Omega(H) - \mu(H) | < \varepsilon$ for all $H \in \mathscr{L}(X)$.

\begin{lemma} \label{push}
If $\mu:\: \mathscr{L}(X) \rightarrow \overline{\mathbb{R}}_{\geq 0}$ is such that $\mu(A) = \lambda(A \cap H_{i_0...i_n})$ for a fixed sequence of positive integers $(i_0, ..., i_n)$, then $\mu(f^{-1}(A)) = \frac{\lambda(A \cap H_{i_1...i_n})}{|f'(H_{i_0})|} $.
\end{lemma}

\begin{proof}
The density $\rho$ of this measure $\mu$ is $\chi_{H_{i_0...i_n}}$. We need to calculate $\mathscr{P}_f(\rho)$. By \eqref{szepPerron} we have the following: \[
\mathscr{P}_f(\rho)(x) = \sum_{z \in f^{-1}(x)} \frac{\chi_{H_{i_0...i_n}}(z)}{|f'(z)|}.
\]
If $\rho(z) \neq 0$, then $z \in H_{i_0}$. Also, if $f^{-1}(x) \cap H_{i_0...i_n} \neq \emptyset$, then $x \in H_{i_1...i_n}$. Since $f$ is injective on $\interior (H_{i_0})$, for almost every $x$ we have that $f^{-1}(x) \cap H_{i_0...i_n}$ has at most one element. These imply the following:\[
\mathscr{P}_f(\rho)(x) = \frac{\chi_{H_{i_1...i_n}}(x)}{|f'(H_{i_0})|}.
\]
\end{proof}

We have that $\omega$ is a linear combination of characteristic functions of finitely many labelled sets. This implies that for some $M \in \mathbb{N}$, for some $\{\eta_k\}_{k=1}^M \in \mathbb{R}_{\geq 0}^M$ and for some $\{(i^{(k)}_0,...,i^{(k)}_{n_k})\}_{k=1}^M$ positive integer sequences of finite length we can represent $\omega$ in the following way:\[
\omega = \sum_{k=1}^M \eta_k \, \chi_{H_{i^{(k)}_0...i^{(k)}_{n_k}}}.
\]

Assuming that $n_k \geq 1$ for all $k$, from Lemma \ref{push} we have that:
\begin{equation} \label{fuj}
\mathscr{P}_f(\omega) = \sum_{k=1}^M \eta_k  \frac{\chi_{H_{i^{(k)}_1...i^{(k)}_{n_k}}}}{|f'(H_{i^{(k)}_0})|}.
\end{equation}

Let $N = \max \{n_k\}_{k=1}^M$. From \eqref{fuj} we have that the pushforward of $\omega$ is still the linear combination of characteristic functions of finitely many labelled sets, however the maximal length of these labels is $N-1$. From the definition of the $\ell$-Markov partition the assumption that $n_k \geq 1$ for all $k$ did not violate generality since the pushforward of the characteristic function of an element of $\mathscr{H}$ is still the linear combination of characteristic functions of finitely many elements of $\mathscr{H}$. By these remarks, for $n>N$ we have that $\mathscr{P}_f^{n}(\omega)$ is a linear combination of characteristic functions of finitely many elements of $\mathscr{H}$. From now on we treat the two cases separately.\medskip

For the first case let us take any $A \in \mathscr{L}(X)$ for which there exists $N \in \mathbb{N}$, such that $\lambda (A \setminus \cup_{i=1}^N H_i ) = 0$. For the second case let us take an arbitrary $A \in \mathscr{L}(X)$. By the first and second cases of Corollary \ref{eloszlas} respectively it is implied that for all such $A$ there is a $K \in \mathbb{N}$, such that if $n>K$ then $|\Omega(f^{-n}(A)) - \Omega(X) \pi(A)| < \varepsilon$. Hence $|\mu(f^{-n}(A)) - \Omega(X) \pi(A)| < 2\varepsilon$. Then $|\mu(X)\pi(A) - \Omega(X) \pi(A)| < \varepsilon \pi(A) \leq \varepsilon$. From these $|\mu(f^{-n}(A)) - \mu(X) \pi(A)| < 3\varepsilon$.
\end{proof}

\subsection{Positive recurrent case}\label{pos_recurrent}

If the Markov chain is irreducible and aperiodic, and has at least one positive recurrent state then all states are positive recurrent. This implies that $\pi(H)=0$ implies $\lambda(H)=0$. Otherwise, if all states are null recurrent or transient, then $\pi(H)=0$ for all $H \in \mathscr{L}(X)$. From now on it is assumed that there is a given dynamical system $(X,f)$ with the properties stated in Theorem \ref{nagyhatar} and that all states of the Markov chain defined by its partition are positive recurrent.

\begin{cor} \label{mix}
A dynamical system which satisfies all conditions of Case 2 of Theorem \ref{nagyhatar} with the corresponding measure $\pi$ is mixing.
\end{cor}

\begin{proof}
Define the measure $\mu(H) = \pi(H \cap B)$ for all $H \in \mathscr{L}(X)$ and some $B \in \mathscr{L}(X)$. From Theorem \ref{nagyhatar}:\[
\lim_{k \rightarrow \infty} \mu(f^{-k}(A)) = \mu(X)\pi(A)
\]
for all $A \in \mathscr{L}(X)$. This can be rewritten as:\[
\lim_{k \rightarrow \infty} \pi(f^{-k}(A) \cap B) = \pi(B)\pi(A).
\]
This is the definition of mixing.
\end{proof}

\begin{lemma} \label{nulltarto}
If $\lambda(H) = 0$, then $\lambda(f^{-1}(H)) = 0$ and $\lambda(f(H)) = 0$. 
\end{lemma}

\begin{proof}
This follows straightforward from that the function is linear with a non-zero slope on every $\interior (H_i) \in \mathscr{H}$, where $\mathscr{H}$ is countable, and that $X \setminus \cup \{\interior (H_i): H_i \in \mathscr{H}\}$ is a countable set. 
\end{proof}

\begin{theorem} \label{invarhalmaz}
If there exists an $E \in \mathscr{L}(X)$, such that $\lambda( f(E) \setminus E ) = 0$, then either $E$ or its complement $X\setminus E = E^c$ has Lebesgue measure zero.
\end{theorem}

\begin{proof}
Proceeding towards a contradiction let us assume that such a set $E$ exists with $\lambda(E)\neq 0 \neq \lambda(E^c)$. We define a measure $\mu(H) = \pi(H \cap E)$. Then we have that $\mu$ is absolutely continuous and Corollary \ref{stoch} implies that $0 < \mu(X) = \pi(E) < \infty$. Then by Theorem \ref{nagyhatar} $\lim_{n\rightarrow \infty}\mu(f^{-n}(E^c)) = \pi(E^c)\cdot \mu(X) = \pi(E^c)\cdot \pi(E) > 0$.\medskip

Furthermore $\lambda( f(E) \setminus E ) = \lambda( f(E) \cap E^c ) = 0$. We have that $f(f^{-1}(E^c) \cap E) \subseteq E^c \cap f(E)$. This implies, that $\lambda(f(f^{-1}(E^c) \cap E)) = 0$. From Lemma \ref{nulltarto} it follows that $\lambda (f^{-1}(E^c) \cap E) = 0$.\medskip

Let us suppose that $\lambda(f^{-n}(E^c) \cap E) = 0$. We show that $\lambda(f^{-n-1}(E^c) \cap E) = 0$. We have that $f^{-n-1}(E^c) = f^{-1}(f^{-n}(E^c) \cap E^c) \cup f^{-1}(f^{-n}(E^c) \cap E)$. We can deduce that $\lambda(f^{-n-1}(E^c) \cap E) \leq \lambda(f^{-1}(f^{-n}(E^c) \cap E^c) \cap E) + \lambda(f^{-1}(f^{-n}(E^c) \cap E) \cap E) \leq \lambda(f^{-1}(E^c) \cap E) = 0$.\medskip

This by induction proves that for all positive $n$ we have $\lambda(f^{-n}(E^c) \cap E) = 0$, which implies that $\mu(f^{-n}(E^c)) = \pi(f^{-n}(E^c) \cap E) = 0$. This contradicts that $\lim_{n\rightarrow \infty}\mu(f^{-n}(E^c)) >0$. 
\end{proof}

In the following theorem we discuss invariant measures which are absolutely continuous with respect ot $\lambda$. According to the most general form of the Radon-Nikdym theorem, these measures have a density $\rho: X \rightarrow \overline{\mathbb{R}}_{\geq 0}$. These measures are not necessarily $\sigma$-finite, and $\sigma$-finiteness is usually assumed when discussing absolutely continuous invariant measures, however having proven Theorem \ref{invarhalmaz} we can describe the non $\sigma$-finite cases as well.

\begin{theorem}
Suppose that $\mu$ is an invariant measure on $\mathscr{L}(X)$ and there exists a density $\rho: X \rightarrow \overline{\mathbb{R}}_{\geq 0}$, such that $\mu(A) = \int_A \rho d\lambda$ for all $A \in \mathscr{L}(X)$. If $\mu(X) = \infty$ and if $\lambda \ll \pi$, then for all $A \in \mathscr{L}(X),\, \lambda(A)\neq 0$ it is true that  $\mu(A)=\infty$.
\end{theorem}

\begin{proof}
We start by stating a lemma.
\begin{lemma}
Let $\gamma$ be the density of $\mu$ and $O = \{x\in X: \gamma(x)=\infty\}$. If $\lambda(O) = 0$ then for all $N$ there exists a measurable set $E\subset X$, such that $\infty > \mu(E) > N$.
\end{lemma}

\begin{proof}
Let us first consider the sequence $\int_{[-n,n]\cap X} \gamma\, d\lambda$. If $n$ goes to infinity the limit must be infinity. If there is an $N$ such that there is no such $n$ for which $\infty > \int_{[-n,n]\cap X} \gamma\, d\lambda > N$, then there exists an $N_0$ for which $\int_{[-N_0,N_0]\cap X} \gamma\, d\lambda = \infty$. Then let $A_k = \{x \in [-N_0,N_0]\cap X:\: \gamma(x)<k\}$. $\lim_{k\rightarrow \infty} \mu(A_k) = \infty$ but $\mu(A_k) < 2N_0k$, which implies that for all $N$ there must be a $K$ such that $\infty > \mu(A_K) > N$.
\end{proof}

If $\mu(X)=\infty$ and for all $N$ there exists a measurable set
$E\subset X$, such that $\infty > \mu(E) > N$, then let $\mu'(H) = \mu(H \cap E)$. Now $\mu(H) = \mu'(H) + \mu(H \cap E^c)$ and $\mu'(X) < \infty$. Since $\mu$ is an invariant measure, for all $H \in \mathscr{L}(X)$ we have that $\lim_{n\rightarrow \infty} \mu(f^{-n}(H)) = \mu(H)$. Also $\mu(H) = \lim_{n\rightarrow \infty} \mu(f^{-n}(H)) \geq \lim_{n\rightarrow \infty} \mu'(f^{-n}(H)) > N\cdot \pi(H)$. In this case the statement is proven. This case also implies the statement if $\mu$ is $\sigma$-finite.\medskip

If there exists an $N$ such that we cannot find a set with a finite $\mu$ measure greater than $N$, then by the previous lemma $\lambda(O) > 0$ for $O = \{x\in X: \gamma(x)=\infty\}$. If $\lambda(O^c) = 0$ then the proof is finished. If $O$ and $O^c$ have non-zero measure, then \ref{invarhalmaz} implies that $\lambda(f(O) \cap O^c) \neq 0$ and $\lambda (f(O^c) \cap O) \neq 0$. Hence, there exists $E \subset f(O) \setminus O$ such that $\lambda(E) \neq 0$ and $\mu(E) < \infty$. Moreover $E$ is a subset of $f(O)$. Since $\lambda(E) \neq 0$ and $E \subseteq f(f^{-1}(E) \cap O)$ then $\lambda (f^{-1}(E)\cap O) \neq 0$ must be true, thus $\mu(f^{-1}(E))=\infty$. This contradicts that $\mu$ is invariant.
\end{proof}

\subsection{Relation to previous results}\label{prev_results}
As stated previously the uniqueness of this approach is in that it describes even the non-$\sigma$-finite absolutely continuous invariant measures of piecewise linear Markov maps, and that in the finite case it shows the convergence of the pushforwards of any absolutely continuous measure. For the latter a very general result is stated in \cite{Bugiel1998}, concerning not only piecewise linear Markov maps. In this subsection an example of a piecewise linear Markov map is given, for which Theorem 3.1 of \cite{Bugiel1998} does not apply, however Theorem \ref{nagyhatar} does.

\begin{exam} \label{savior}
Take the random walk on $\mathbb{Z}^+$ where from $n \geq 3$ we step on $n+1, n, n-1$ or $n-2$ with the probability of each being $\nicefrac{1}{4}$, and from one or two we step onto an element of $\{1,2,3,4\}$ with the probability of each being $\nicefrac{1}{4}$. Take a function which defines a transition matrix according to Definition \ref{Transmatrix} defining this Markov chain.
\begin{equation}
f(x) =
   \begin{cases} 
      4(x-\lfloor x \rfloor) & \text{if } 0 \leq x < 3, \\
      4x - 3\lfloor x \rfloor -2 & \text{if } 3 \leq x.
   \end{cases}
\end{equation}
For a plot of this function see figure \ref{kek}. The set of integer intervals $\{I_k \}$ form an expansive $\ell$-Markov partition of the domain of this function and Theorem \ref{nagyhatar} applies to $(\mathbb{R}_{\geq 0}, f)$ in a straightforward way.
\end{exam}
\begin{figure}[h]
\includegraphics[width= 5cm]{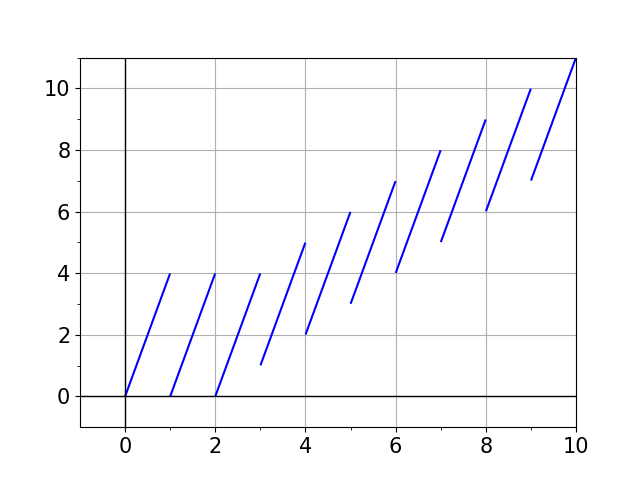}
\centering
\caption{Plot of $f(x)$ from example \ref{savior}.}
\label{kek}
\end{figure}

Let $k(r) = (k_0,k_1,...k_{r-1})$ an $r$ long multi-index, we denote $I_{k_0k_1...k_{r-1}}$ as $I_{k(r)}$. Let the set of all  $r$ long multi-indeces, such that the indeces are positive integers be $K^r$. We define a number of objects used in the statement of Theorem 3.1 of \cite{Bugiel1998} and explain what they are in our special case.
\begin{equation}
I_{k(r)}= I_{k_0} \cap f^{-1}(I_{k_1}) \cap ... \cap f^{-r-1}(I_{k_{r-1}})
\end{equation}
\begin{equation}
f_{k(r)} := (f^r)|_{I_{k(r)}}, \quad J_{k(r)} := f_{k(r)}(I_{k(r)}) = f(I_{k_{r-1}}).
\end{equation}
\begin{equation}
m_{k(r)}(A) := \lambda((f_{k(r)}^{-1})(A)).
\end{equation}
\begin{equation}
\sigma_{k(r)} :=
\begin{cases}
      \frac{dm_{k(r)}}{d\lambda} & \text{on } J_{k(r)}, \\
      0 & \text{otherwise }
\end{cases},\quad
\tilde{\sigma}_{k(r)} := \frac{\sigma_{k(r)}}{\lambda(I_{k(r)})}.
\end{equation}
In our special case where the function is linear on all elements of the partition we have the following:
\begin{equation}
m_{k(r)}(A) = \frac{\lambda(A \cap J_{k(r)})}{|(f^r)'(I_{k(r)})|} = \frac{\lambda(I_{k(r)})}{\lambda(J_{k(r)})} \lambda(A \cap J_{k(r)}), \quad \tilde{\sigma}_{k(r)} = \frac{1}{\lambda(J_{k(r)})} \chi_{J_{k(r)}}.
\end{equation}
The last two expressions needed are the following:
\begin{equation}
g_{k(r)} := \sum_{s(r)} \tilde{\sigma}_{s(r)} \int_{I_{s(r)}} \tilde{\sigma}_{k(r)} d\lambda, \quad u_{r}(x) := \inf \{g_{k(r)}(x) : I_{k(r)} \neq \emptyset\}.
\end{equation}
Where the sum and supremum is taken on all $r$-long multi-indices. In our special case we have the following:
\begin{equation} \label{t8}
g_{k(r)}(x) = \sum_{s(r)} \frac{\chi_{J_{s(r)}}(x) \lambda(I_{s(r)} \cap J_{k(r)})}{\lambda(J_{s(r)}) \lambda(J_{k(r)})}.
\end{equation}
Any element of this sum is supported on $J_{s(r)}$. Only those elements are non-zero, for which $\lambda (I_{s(r)} \cap J_{k(r)}) \neq 0$. Since $J_{k(r)} = f(I_{k_{r-1}})$, from the definition of an $\ell$-Markov partition there is a subset $S \subset K^r$, such that $J_{k(r)} = \bigcup_{s(r) \in S} I_{s(r)}$. Thus $\bigcup_{s(r) \in S} J_{s(r)} = f^r(J_{k(r)})$. It follows, that $g$ in \eqref{t8} is a function supported on $f^r(J_{k(r)}) = f^{r+1}(I_{k_{r-1}})$. In the case of $f$ in Example \ref{savior} we have that for all $x$ and for all $r$ there exists $k(r)$ such that $x \notin f^{r+1}(I_{k_{r-1}})$. Namely we choose $I_{k(r)} \neq \emptyset$ such that $k_{r-1} > \lceil x \rceil + 2r + 4$ applies. Thus for all $r$ and $x$ we have $u_r(x) = 0$. One of the conditions of Theorem 3.1 in \cite{Bugiel1998} is that there exists $r>1$ such that $||u_r||_{L^1} > 0$. This does not apply in this case.

\section{Applications to our system} \label{sec4}
In this section, we compute the unique absolutely continuous invariant measure of $(\mathbb{R}_{\geq 0}, T)$, the existence of which is implied by \ref{nagyhatar}. We also show how is the solution of the original Schweitzer problem implied by our results in Corollary \ref{solution}.

\begin{theorem}
The closed integer intervals form an expansive $\ell$-Markov partition of $(\mathbb{R}_{\geq 0}, T)$.
\end{theorem}

\begin{proof}
The properties established in Lemma \ref{Rendszer} satisfy all conditions needed.
\end{proof}

\begin{theorem}
Consider the closed integer intervals as the $\ell$-Markov partition of our system. Then the transition matrix $P = (p_{ij})$ of $(\mathbb{R}_{\geq 0}, T)$ is the following:
\begin{equation}
p_{ij} =
\begin{cases}
		\frac{1}{4^{ \ell (i) }+1} & \text{if } j \in [1, 4^{\ell(i)}+1]\\
		0 & \text{otherwise. } 
	\end{cases}
\end{equation}
\end{theorem}

This Markov chain is irreducible and aperiodic. By finding a stationary distribution with positive coordinates, we will also show that all states are positive recurrent.

\begin{defin}
We define the $n$th step of the function. It is a closed interval. It is the union of closed integer intervals, the $T$ image of which are identical. Explicitly:\[
L_1 = [0, 4] \quad
L_n = [4^{n-1}, 4^n].
\]
Let $\mathcal{L}_n$ be the set of indices of those integer intervals $I_k$ for which $I_k \subset L_n$. 

\end{defin}

Let us construct the transition matrix $S$ of a new Markov chain the states of which will represent the "steps" of the function $T$. It is interesting to study it a bit more generally. Let us introduce a parameter $E\in \mathbb{Z}_{\geq 2}$. The matrix $S$ will describe our system if $E=4$. For $k>1$, let: 
\begin{equation} \label{Lambda}
\Lambda(1)=E \text{ and } \Lambda(k)=(E-1)E^{k-1}.
\end{equation}
If $E=4$ we have that $\Lambda(n) = \lambda(L_n)$. We define $S=(s_{ij})$ by:
\begin{equation} \label{Smatrix}
s_{ij} =
\begin{cases}
		\frac{\Lambda(j)}{E^i + 1} & \text{if } j \leq i,\\
		\frac{1}{E^{i}+1} & \text{if } j = i+1,\\
		0 & \text{otherwise.} 
	\end{cases}
\end{equation}

The reader might wonder why this chain is of particular interest to us. Suppose that we have a measure $\mu$ with finite density $\rho$, which is constant on every integer interval. Since the image of all integer intervals in a step is identical, to calculate the density of the pushforward of $\mu$, we only need to know $\mu(L_k)$ for every $k$. Let the constants of $\rho$ be stored in a  countably infinite dimensional row vector $v = (v_1, v_2, ...)$, as in Theorem \ref{perron}. From Theorem \ref{perron} we know that the constants of the pushforward of $\mu$ are in the vector $vP$. Let us define the vector $w^{(v)} = (w^{(v)}_1, w^{(v)}_2, ...)$ such that $w^{(v)}_k = \sum_{i \in \mathcal{L}_k}v_i = \mu(L_k)$. Hence we have that if $E=4$, then $\mu(T^{-1}(L_k)) = (w^{(v)}S)_k$, thus $w^{(v)}S = w^{(vP)}$. If for an infinite dimensional row vector $p$ we have $pP = p$, meaning that $p$ is a stationary distribution of $P$, then $w^{(p)}S = w^{(p)}$, so $w^{(p)}$ is a stationary distribution of $S$.

\begin{theorem} \label{Qstac}
The only possible $q = (q_1, q_2, ...)$ stochastic, countably infinite dimensional row vector, for which $q = qS$ is given by:\[
q_1 =\frac{1}{2}, \quad q_{k+1}=\frac{E}{E^k+1}q_{k}, \quad k\in \mathbb{Z}^+.
\]
This is equivalent to saying that it is the only stationary distribution of the Markov chain defined by $S$.
\end{theorem}

\begin{proof}
Let us pick any stochastic infinite dimensional row vector $q = (q_1, q_2, ...)$. Let $qS^n = q^{(n)} = (q^{(n)}_1, q^{(n)}_2,...)$. We have the following recursion for the first coordinate:
\begin{equation}\label{Recur0}
q_1^{(n)} = \sum_{i=1}^\infty \frac{\Lambda(1)}{E^i + 1}q_i^{(n-1)}.
\end{equation}
For $k \geq 2$ we have the following:
\begin{equation} \label{Recur1}
q_k^{(n)}=\frac{1}{E^{k-1}+1}q_{k-1}^{(n-1)}+\sum_{i=k}^{\infty} \frac{\Lambda(k)}{E^i+1}q_i^{(n-1)}.
\end{equation}
We can deduce from \eqref{Recur0} and \eqref{Recur1} the following formulae for the coordinates of $q_k^{(n)}$:
\begin{equation}\label{Recur02}
q_2^{(n)} = \frac{\Lambda(2)}{\Lambda(1)}\Big(q_1^{(n)} - \frac{\Lambda(1)}{E+1}q_1^{(n-1)} \Big) + \frac{1}{E+1}q_1^{(n-1)},
\end{equation}
\begin{equation} \label{Recur2}
q_{k+1}^{(n)}=\frac{\Lambda(k+1)}{\Lambda(k)}\Big( q_{k}^{(n)}-\frac{1}{E^{k-1}+1}q_{k-1}^{(n-1)}-\frac{\Lambda(k)}{E^k+1}q_{k}^{(n-1)} \Big)+\frac{1}{E^k+1}q_k^{(n-1)} \text{ for } k\geq 2.
\end{equation}
For a stationary distribution $q_k^{(n)}=q_k^{(m)}$, so we can omit the upper index in the recursion. If $k>1$, then by \eqref{Lambda} and \eqref{Recur2} for a stationary distribution:
\begin{equation} \label{Recur3}
q_{k+1}=\frac{E^k+E+1}{E^k+1}q_k-\frac{E}{E^{k-1}+1}q_{k-1}.
\end{equation}
Let us sum \eqref{Recur3} from $k=2$ to $k=n-1$. We get the following:
\[
\sum_{k=3}^n q_k = \sum_{k=3}^{n-1} q_k + \frac{E}{E^{n-1}}q_{n-1}+q_2-\frac{E}{E+1}q_1.
\]
Let $n$ go to infinity. Since the vector is stochastic its coordinates form an absolutely convergent series:
\[
\sum_{k=3}^{\infty} q_k = \sum_{k=3}^{\infty} q_k +q_2-\frac{E}{E+1}q_1.
\]
\[
\frac{E}{E+1}q_1 =  q_2. 
\]

From here the recursion in the proposition is satisfied for $k=1$. Let us suppose that it is satisfied for $k=n-1$, thus $\frac{E}{E^{n-1} + 1}q_{n-1} = q_n$. We substitute this into \eqref{Recur3}. We obtain the following:
\begin{equation} \label{itt}
q_{n+1} = \frac{E^n+E+1}{E^n+1}q_n-\frac{E}{E^{n-1}+1}q_{n-1} = \frac{E^n+E+1}{E^n+1}q_n- q_{n} = \frac{E}{E^{n} + 1}q_{n}.
\end{equation}
Hence the recursion is satisfied for $k=n$, thus the recursion follows by induction for all $k$. We only need to find the value of $q_1$. For this we sum the elements of the distribution in terms of $q_1$.\medskip

Let $q_{k+1}=a_{k+1}q_1$. Since by \eqref{itt} we have proven the recursion of $q_k$, it follows that:
\[
a_1 = 1,\quad a_{k+1}= \frac{E}{E^k+1} a_k=\frac{E^k}{\prod_{i=1}^{k}E^i+1}.
\]

Next we show by induction that $\sum_{k=1}^n a_k=2-\frac{a_n}{E^{n-1}}$. For $n=1$ the statement is $1 = 2 - 1$. If the statement is true for $n \in \mathbb{Z}^+$, then: 
$$\sum_{k=1}^{n+1}a_k = 2-\frac{a_n}{E^{n-1}} + a_{n+1} = 2+ \Big(1 - \frac{E^n + 1}{E^n}\Big)a_{n+1}.$$
Thus it is true for $n+1$. This implies that $\sum_{k=1}^\infty a_k = 2$. Since $\sum_{k=1}^\infty q_k = 1$ and $q_k = a_kq_1$ we have that the value of $q_1$ must be $\nicefrac{1}{2}$ independent of $E$.\medskip

The only remaining statement to verify is that the found $q = (q_0,q_1,...)$ truly satisfies $qS=q$. We return to the notation that $q^{(n)} = q^{(n-1)}S$. We show coordinatewise that if $q^{(n-1)} = q$ then $q^{(n)} = q$.\medskip

For the first coordinate by \eqref{Recur0}, we have that: 
\begin{equation}
q_1^{(n)} = \sum_{i=1}^{\infty} \frac{\Lambda(1)}{E^{i}+1} q_i = \sum_{i=1}^{\infty} \frac{E}{E^{i}+1} a_i q_1 =q_1 \sum_{i=1}^{\infty} a_{i+1} = q_1(2-a_1) = q_1.
\end{equation}

For the second coordinate, by \eqref{Recur02} we have the following:
\begin{equation}
q_2^{(n)} = \frac{(E-1)E}{E}\Big(q_1^{(n)} - \frac{E}{E+1}q_1^{(n-1)} \Big) + \frac{1}{E+1}q_1^{(n-1)} = \frac{E}{E+1}q_1^{(n-1)} = q_2^{(n-1)}.
\end{equation}

For the remaining coordinates let us assume that for all $i \leq k$ for some $2\leq k$ we have that $q_i^{(n)}=q_i^{(n-1)} = q_i$. We prove that in this case $q_{k+1}^{(n)}=q_{k+1}^{(n-1)} = q_{k+1}$. By these assumptions from \eqref{Recur2} we arrive to a formula which is almost identical to \eqref{Recur3}. We get the following:
\begin{equation}
q_{k+1}^{(n)}=\frac{E^k+E+1}{E^k+1}q_k^{(n-1)}-\frac{E}{E^{k-1}+1}q_{k-1}^{(n-1)}.
\end{equation}
It has already been established that $q$ satisfies this recursion. Thus $q$ is truly the only stationary distribution of this Markov chain.
\end{proof}

From here we can find the only possible stationary distribution of $P$ as well. Let us generalize it as well.

\begin{defin}
Let $E\in \mathbb{Z}_{\geq 2}$. Let us define the function $\ell:\: \mathbb{R}_{\geq 0} \times \mathbb{Z}_{\geq 2} \rightarrow \mathbb{N}$ as follows:\[
\ell(x,E) = \inf\{n\in \mathbb{Z}^+:\: x \leq E^n \}.
\]
If $E=4$, then the original definition of $\ell$ still applies. From now on, by $\ell(x)$ we mean $\ell(x,4)$.  Let us accordingly generalize the matrix $P = (p_{ij})$.
\begin{equation} \label{transmatrix}
p_{ij} =
\begin{cases}
		\frac{1}{E^{ \ell (i, E) }+1} & \text{if } j \in [1, E^{\ell(i,E)}+1]\\
		0 & \text{otherwise. }
	\end{cases}
\end{equation}
\end{defin}

\begin{defin}
Let $L_1^E = [0,E]$ and $L_n = [E^{n-1}, E^n]$, moreover let $\mathcal{L}_n^E =\{k\in \mathbb{Z}^+: I_k \subseteq L_n\}$.
\end{defin}

These definitions were constructed such that $i \in \mathcal{L}_{\ell(i,E)}^E$ and that $\lambda(L_n^E) = \Lambda(n)$ for any $E \in \mathbb{Z}_{\geq 2}$.

\begin{lemma} \label{c8}
For fixed $i,n \in \mathbb{Z}^+$ the following equation holds:
\begin{equation}
\sum_{j \in \mathcal{L}_n^E} p_{ij} = s_{\ell (i, E)n}.
\end{equation}
\end{lemma}

\begin{proof}
For this part recall the definition of $S$ by \eqref{Smatrix}. For $n > \ell(i,E) + 1$ we always have that $j > E^{\ell(i,E)}+1$, thus $\sum_{j \in \mathcal{L}_n^E} p_{ij} = 0 = s_{\ell (i, E)n}$. For $n = \ell(i,E) + 1$ we only get a non zero transition probability $p_{ij}$ if $j = E^{\ell(i,E)}+1$, thus $\sum_{j \in \mathcal{L}_n^E} p_{ij} = \frac{1}{E^{\ell(i,E)}+1} = s_{\ell (i, E)n}$. For $n \leq \ell(i,E)$ we have that for all $j \in \mathcal{L}_n^E$ the value of $p_{ij}$ is the same, thus $\sum_{j \in \mathcal{L}_n^E} p_{ij} = \frac{\Lambda(n)}{E^{\ell(i,E)}+1} = s_{\ell (i, E)n}$.
\end{proof}

\begin{lemma} \label{qazq}
Suppose that for a stochastic countably infinite dimensional row vector $p = (p_1, p_2, ...)$ we have that $p = pP$. If $q_n' = \sum_{i \in \mathcal{L}_n^E} p_i$, then for $q' = (q_1', q'_2...)$ we have that $q' = q$.
\end{lemma}
\begin{proof}
We have that $q'$ is a stochastic countably infinite dimensional row vector, thus $q' = q$ is equivalent to $q'S = q'$ by Theorem \ref{Qstac}. We show that $q'$ is a stationary distribution of $S$. Since $p$ is a stationary distribution of $P$ we have the following:
\begin{equation} \label{ineqout}
p_j = \sum_{i=1}^\infty p_i p_{ij}.
\end{equation}
Thus for all $n \in \mathbb{Z}^+$:
\begin{equation} \label{3sum}
\sum_{i \in \mathcal{L}_n^E} p_i = \sum_{i \in \mathcal{L}_n^E} \sum_{j=1}^\infty p_j p_{ji} = \sum_{i \in \mathcal{L}_n^E} \sum_{k=1}^\infty \sum_{j \in \mathcal{L}_k^E}p_j p_{ji}.
\end{equation}
If $a,b \in \mathcal{L}_k^E$ for some $k \in \mathbb{Z}^+$, then $p_{ai} = p_{bi}$ for all $i$. From this and from the definition of $q'$ we obtain the following from \eqref{3sum}:
\begin{equation}
 q'_n = \sum_{i \in \mathcal{L}_n^E} p_i = \sum_{i \in \mathcal{L}_n^E} \sum_{k=1}^\infty p_{E^ki} \sum_{j \in \mathcal{L}_k^E}p_j = \sum_{i \in \mathcal{L}_n^E} \sum_{k=1}^\infty p_{E^ki} q'_k.
\end{equation}
Since all sums are absolutely convergent we can flip the order of them. By Lemma \ref{c8}:
\begin{equation}
q'_n = \sum_{k=1}^\infty q'_k \sum_{i \in \mathcal{L}_n^E} p_{E^ki} = \sum_{k=1}^\infty q'_k s_{kn}. 
\end{equation}
By this $q' = q'S$, thus $q=q'$.
\end{proof}

\begin{theorem}
The only possible stochastic infinite dimensional row vector $p = (p_1, p_2, ...)$, for which $p = pP$ is given by:\[
p_k = 
\begin{cases}
 \frac{1}{2E} & \text{if } k \in \{1,...,E+1\}\\
 \frac{1}{E^{n}}q_{n} & \text{otherwise, if } n>1 \text{ and } k \in \{E^{n-1}+2,..., E^n+1\}. 
\end{cases}
\]
This is equivalent to saying that it is the only stationary distribution of the Markov chain defined by $P$.
\end{theorem}

\begin{proof}
We suppose, that $p$ is a stationary distribution, thus \eqref{ineqout} is true. We can deduce by \eqref{transmatrix} that if $a,b \in \{1,...,E+1\}$ or $a,b \in \{E^{k-1}+2, E^k+1\}$ for some $k \in \mathbb{Z}^+$, then $p_a = p_b$ must be true since $p_{ia}=p_{ib}$ for all $i$ in this case. By Lemma \ref{qazq} we know that:
\begin{equation} \label{a8}
q_n = \sum_{k=E^{n-1}+1}^{E^n} p_k = \sum_{k \in \mathcal{L}_n^E} p_k.
\end{equation}
By this $p_1 =...=p_{E+1} = \frac{1}{2E}$.\medskip

Now let us consider a step, other than the first one. Let $p'=p_k$ where $k=E^{n-1}+1$ for some $n \in \mathbb{Z}^+$. Let $p = p_{k+1} = ... = p_{E^{n}}$. From \eqref{transmatrix} we have that for $k=E^{n-1}+1$ and any $i \in \mathbb{Z}^+$ the transition probabilities $p_{i(k+1)} = p_{ik}$ if and only if $i \notin \mathcal{L}_{n-1}^E$. If $i \in \mathcal{L}_{n-1}^E$, then $p_{ik} = \frac{1}{E^{n-1}+1}$ and $p_{i(k+1)}=0$. By this, and by applying \eqref{ineqout} we obtain:
\begin{equation}
p' - p = \sum_{i=1}^\infty p_i p_{ik} - \sum_{i=1}^\infty p_i p_{i(k+1)} = \sum_{i \in \mathcal{L}_{n-1}^E} \frac{p_{i}}{E^{n-1}+1} \text{.}
\end{equation}
Hence by \eqref{a8}:
\begin{equation} \label{prtty1}
p'=p+\frac{1}{E^{n-1}+1}q_{n-1}.
\end{equation}
By \eqref{prtty1}, by \eqref{a8} and by Theorem \ref{Qstac} we obtain:
\begin{equation}
q_n=\frac{E}{E^{n-1}+1}q_{n-1}=(E-1)E^{n-1}p + \frac{1}{E^{n-1}+1}q_{n-1}.
\end{equation}
By simple algebra we get $p = \frac{1}{E^{n-1}(E^{n-1}+1)}q_{n-1} = \frac{1}{E^n}q_n$ and $p'=\frac{1}{E^{n-1}}q_{n-1}$. If $n > 2$, then by Theorem \ref{Qstac} $p' = \frac{1}{E^{n-2}(E^{n-2}+1)}q_{n-2} = \frac{1}{E^{n-1}}q_{n-1}$, if $n=2$, then $p'=\frac{1}{2E}$. Now we have proven that the stationary distribution of the Markov chain has to be the one we have described in our statement. We also verify that it is really a stationary distribution.\medskip

Let us again use the notation that $p^{(n-1)} = p$ and $p^{(n)} = p^{(n-1)}P$. First assume that $k \in \{1,...,E+1\}$. From \eqref{transmatrix} and \eqref{Recur0} we have the following:
$$p_k^{(n)} = \sum_{i=1}^\infty \frac{1}{E^{\ell(i)}+1} p_i^{(n-1)} = \sum_{i=1}^\infty \frac{1}{E^{i}+1}q_i = \frac{1}{E} q_1 = \frac{1}{2E}.$$
For a given $n > 1$ by \eqref{transmatrix} and \eqref{Recur1} we have for all $k \in \{E^{n-1}+2,...,E^n + 1\}$ the following:
\begin{multline}
p_k^{(n)} = \sum_{i = E^{n-1}+1}^\infty \frac{1}{E^{\ell(i)}+1} p_i^{(n-1)} = \sum_{i=n}^\infty \frac{1}{E^i + 1} q_i = \frac{1}{\Lambda(n)} \Big( q_n - \frac{1}{E^{n-1}+1} q_{n-1}\Big) \\
= \frac{1}{(E-1)E^{n-1}} \Big( \frac{E}{E^{n-1} + 1} - \frac{1}{E^{n-1} + 1}\Big)q_{n-1} = \frac{1}{E^{n-1}(E^{n-1}+1)}q_{n-1} = p_k^{(n-1)}.
\end{multline}
This proves the whole statement.
\end{proof}

According to Theorem \ref{nagyhatar} this stationary distribution, for $E=4$ gives an invariant measure $\pi:\:\mathscr{L}(\mathbb{R}_{\geq0}) \rightarrow \mathbb{R}_{\geq 0}$ for the system $(\mathbb{R}_{\geq 0}, T)$. For this measure if $\mu$ is an absolutely continuous measure such that $\mu(\mathbb{R}_{\geq 0})<\infty$, then $\lim_{n\rightarrow \infty}\mu(f^{-n}(H)) =\mu(\mathbb{R}_{\geq 0})\pi(H)$ for all $H \in \mathscr{L}( \mathbb{R}_{\geq 0} )$. This also means that the system is mixing from Corollary \ref{mix}. Since $\pi$ is such that $\pi(H) = 0$ implies $\lambda(H) = 0$ for measurable subsets of the domain, from theorems of Section \ref{sec3} the only absolutely continuous invariant measures of the system are those which are constant multiples of $\pi$, or if we consider non-$\sigma$-finite measures, infinite for all non-null sets.

\begin{cor} \label{solution}
The orbit of almost every point is dense in $\mathbb{R}_{\geq 0}$, thus for almost every $x \in \mathbb{R}_{\geq 0}$ we have that $\liminf_{k\rightarrow \infty}(f^k(x)) = 0$ and $\limsup_{k\rightarrow \infty}(f^k(x)) = \infty$.
\end{cor}
\begin{proof}
The absolutely continuous measure $\pi$ satisfies all conditions needed for Theorem \ref{mix}, thus the system is mixing. Thus for any $A \in \mathscr{L}(\mathbb{R}_{\geq 0})$ for all $I_n$ there exists an $m \in \mathbb{N}$ such that $\pi (f^{-m}(I_n) \cap A) \geq \nicefrac{1}{2}(\pi(A)\pi(I_n))$. Let $A_0	 = A,\: A_k = A_{k-1} \setminus f^{-m_k}(I_n)$ where: $$m_k = \min \{m \in \mathbb{N}: \pi (f^{-m}(I_n) \cap A_{k-1}) > \frac{1}{2}(\pi(A_{k-1})\pi(I_n))\}.$$
Since $\pi(\mathbb{R}_{\geq 0})=1$ we have that $\pi(A_k) < \frac{\pi(I_n)}{2}\pi(A_{k-1}) < \pi(A_{k-1})$, thus $\lim_{k \rightarrow \infty}(\pi(A_k))=0$. Let $H \subset A$ be the set of those points whose orbit never intersects $I_n$. We have that $H \subseteq \bigcap_{k=0}^\infty A_k$ thus $\pi(H) = 0$ implying that $\lambda(H) = 0$.
\end{proof}

\addcontentsline{toc}{section}{Acknowledgements}
\section*{Acknowledgements}
I would like to express my deepest gratitude to my advisor Zoltán Buczolich, for proposing these intriguing problems and providing guidance throughout the preparation of this paper. I would also like to thank Henk Bruin for his insights on the paper and for suggesting relevant literature.

\bibliographystyle{unsrt}
\addcontentsline{toc}{section}{References}
\bibliography{hmm}

\end{document}